\documentclass{amsart}
\usepackage{graphicx}
\newtheorem{theorem}{Theorem}[section]

\newtheorem{problem}[theorem]{Problem}
\newtheorem{corollary}[theorem]{Corollary}
\newtheorem{conjecture}[theorem]{Conjecture}

\theoremstyle{definition}

\begin{document}

\title[Coexistence of coiled surfaces and spanning surfaces]{Coexistence of coiled surfaces and spanning surfaces for knots and links}

\author{Makoto Ozawa}
\address{Department of Natural Sciences, Faculty of Arts and Sciences, Komazawa University, 1-23-1 Komazawa, Setagaya-ku, Tokyo, 154-8525, Japan}
\email{w3c@komazawa-u.ac.jp}
\thanks{This work was supported by JSPS KAKENHI Grant Number 23540105.
%The author is partially supported by Grant-in-Aid for Scientific Research (C) (No. 23540105), The Ministry of Education, Culture, Sports, Science and Technology, Japan
}



\begin{abstract}
It is a well-known procedure for constructing a torus knot or link that first we prepare an unknotted torus and meridian disks in the complementary solid tori of it, and second smooth the intersections of the boundary of meridian disks uniformly. Then we obtain a torus knot or link on the unknotted torus and its Seifert surface made of meridian disks.
In the present paper, we generalize this procedure by a closed fake surface and show that the resultant two surfaces obtained by smoothing triple points uniformly are essential.
We also show that a knot obtained by this procedure satisfies the Neuwirth conjecture, and the distance of two boundary slopes for the knot is equal to the number of triple points of the closed fake surface.
\end{abstract}

\maketitle

\section{Introduction}

\subsection{The Neuwirth conjecture}

There are not so many geometrical properties satisfied by all non-trivial knots.
Any knot bounds a minimal (and hence incompressible) Seifert surface \cite{FP}, \cite{S}, and for any non-trivial knot there exists a properly embedded separating, orientable, incompressible, boundary incompressible and not boundary parallel surface in the exterior of the knot \cite{CS}.
The following conjecture asserts that any non-trivial knot can be embedded in a closed surface, similarly to the way a torus knot can be embedded in an unknotted torus.

\begin{conjecture}[Neuwirth Conjecture, \cite{N}]
For any non-trivial knot $K$ in the 3-sphere, there exists a closed surface $F$ containing $K$ non-separatively such that $F$ is essential in the exterior of $K$.
\end{conjecture}

Recent results on the Neuwirth conjecture can be seen in \cite{OR}.
Since the Neuwirth conjecture originated in torus knots, we go back to the construction of torus knots in the next subsection.

\subsection{A procedure for constructing torus knots and links}

The following is a well-known procedure for constructing a torus knot or link \cite{L}.
Let $T$ be an unknotted torus in the 3-sphere $S^3$ which decomposes $S^3$ into two solid tori $V_1$ and $V_2$.
Take $p$ mutually disjoint meridian disks $D_1$ of $V_1$ and $q$ mutually disjoint meridian disks $D_2$ of $V_2$.
If we smooth the intersections of $\partial D_1$ and $\partial D_2$ uniformly in $T$, then we can obtain a torus knot or link $K$ of type $(p,q)$.
For each point of $\partial D_1\cap \partial D_2$, we add two triangle regions along this smoothing to $D_1\cup D_2$, and then we obtain a Seifert surface $F_v$ for $K$.
We remark that by the construction, $\chi(F_v)=|D_1|+|D_2|-|\partial D_1\cap \partial D_2|=p+q-pq$ and when $K$ is a knot, $g(F_v)=(p-1)(q-1)/2=g(K)$.
We also have cabling annuli $F_h=T\cap E(K)$, where $E(K)$ denotes the exterior of $K$ in $S^3$.
Moreover when $K$ is a knot, we have $\Delta(\partial F_v, \partial F_h)=|\partial D_1\cap \partial D_2|=pq$, where $\Delta(*,*)$ denotes the distance between two boundary slopes.
We note that $F_v$ is orientable and when $K$ is a knot, $F_h$ is connected.

\subsection{From closed fake surfaces to dual surfaces}

We define three subsets of $\mathbb{R}^3$ as below.

\begin{enumerate}
\item $\Sigma_1=\{ (x,y,z)\in \mathbb{R}^3 | z=0 \}$
\item $\Sigma_2=\{ (x,y,z)\in \mathbb{R}^3 | y=0, z\ge 0 \}$
\item $\Sigma_3=\{ (x,y,z)\in \mathbb{R}^3 | x=0, z\le 0 \}$
\end{enumerate}

A finite 2-polyhedron $P$ is called a {\em closed fake surface} \cite{I} if each of its points has a neighborhood homeomorphic to one of the followings (Figure \ref{type}).
\begin{description}
\item[Type 1] $\Sigma_1$
\item[Type 2] $\Sigma_1\cup \Sigma_2$
\item[Type 3] $\Sigma_1\cup\Sigma_2\cup\Sigma_3$
\end{description}
We will refer to points in closed fake surfaces as points of Type 1, 2 and 3 respectively depending on which of the above three neighborhoods they have.
By $P'$, we shall denote the set of points of Type 2 or 3.
By $P''$, we denote the set of points of Type 3.
A closed fake surface is {\em orientable} if each component of $P-P'$ is orientable.

\begin{figure}[htbp]
	\begin{center}
	\begin{tabular}{ccc}
	\includegraphics[trim=0mm 0mm 0mm 0mm, width=.3\linewidth]{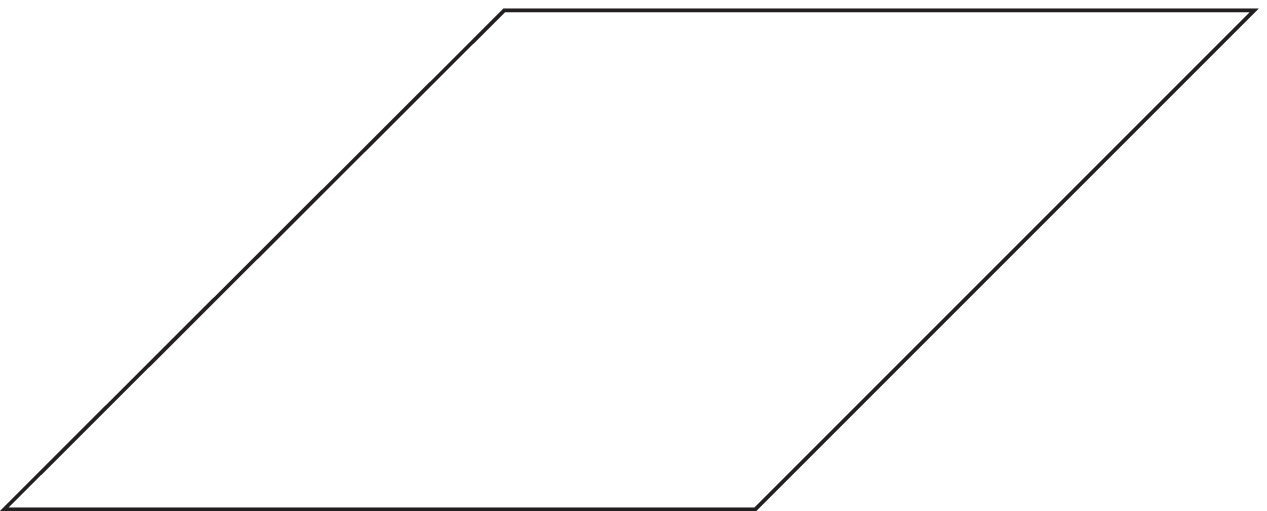} &
	\includegraphics[trim=0mm 0mm 0mm 0mm, width=.3\linewidth]{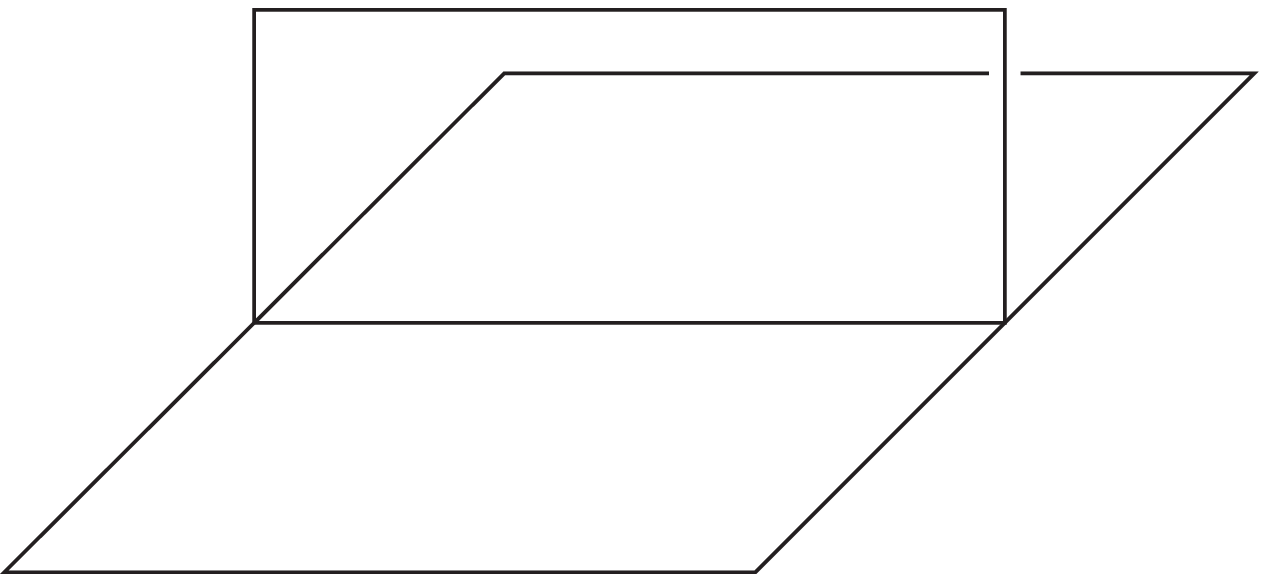} &
	\includegraphics[trim=0mm 0mm 0mm 0mm, width=.3\linewidth]{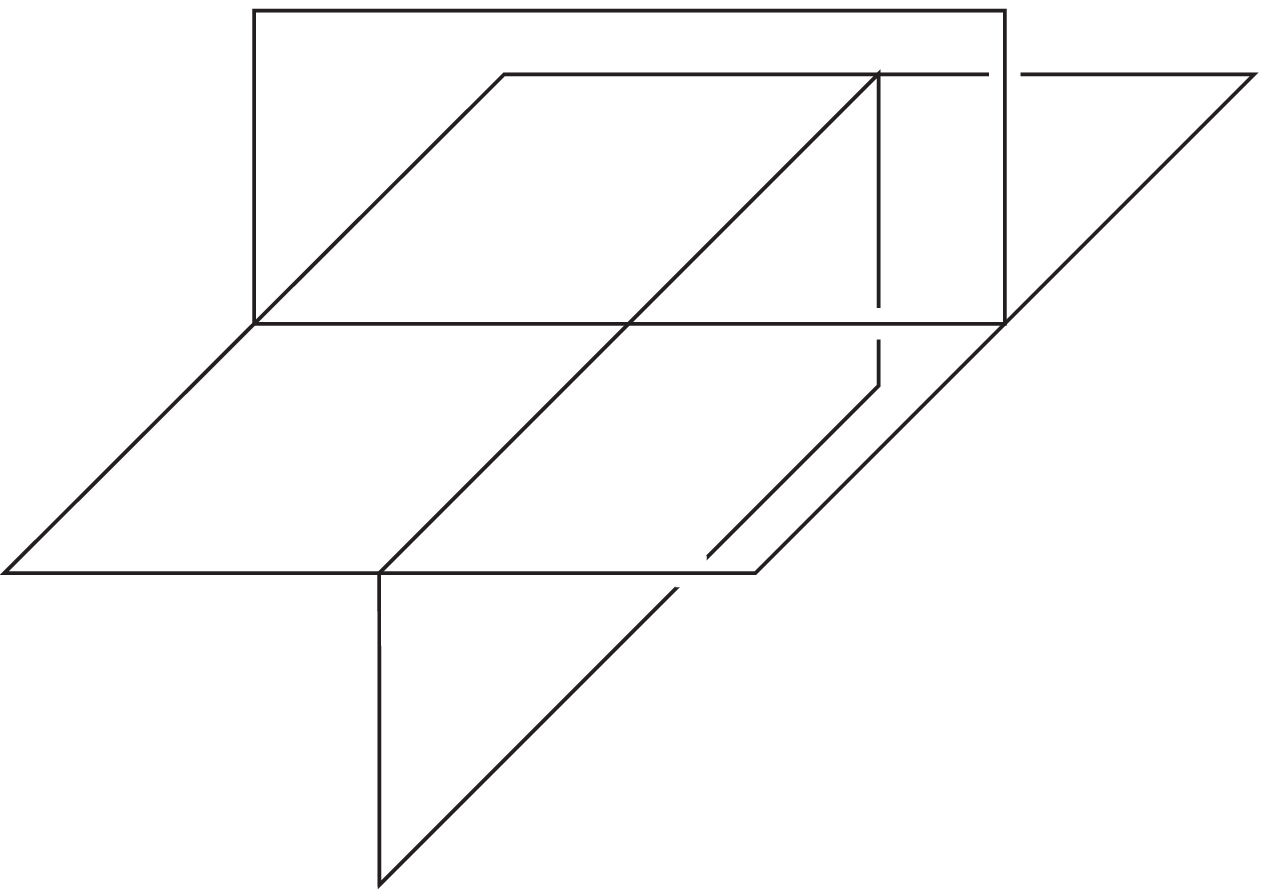} \\
	Type 1 & Type 2 & Type 3
	\end{tabular}
	\end{center}
	\caption{Local neighborhoods of a closed fake surface}
	\label{type}
\end{figure}

We say that a closed fake surface $P$ embedded in $S^3$ has a {\em vertical-horizontal decomposition} $P=P_v\cup P_h$ if $P_h$ is closed subsurfaces of $P$ which corresponds to $(\Sigma_1,\mathbb{R}^3)$ at each neighborhood of points of Type 2 or 3, and $P_v$ is subsurfaces of $P$ which corresponds to $(\Sigma_2,\mathbb{R}^3)$ or $(\Sigma_3,\mathbb{R}^3)$ at each neighborhood of points of Type 2 or 3.
When $P'=\emptyset$, we define $P_h=P$ and $P_v=\emptyset$.

As in the procedure for constructing torus knots and links, for a closed fake surface $P$ with the vertical-horizontal decomposition $P=P_v\cup P_h$, we obtain a knot or link $K$ from $P'$ and the vertical surfaces $F_v$ and the horizontal surfaces $F_h$ from $P_v$ and $P_h$ respectively by smoothing $P$ uniformly as follows.
For each neighborhood of a point of Type 3, we add two triangle regions $\{ (x,y,z)\in \mathbb{R}^3| xy\ge 0, |x+y|\le 1 \}$ to $P_v$.
Then we obtain surfaces from $P_v$ and call it the {\em vertical surfaces} which is denoted by $F_v$.
We note that $\chi(F_v)=\chi(P_v)-|P''|$.
The boundary of $F_v$ consists of disjoint simple closed curves in $P_h$, namely a knot or link, and we denote it by $K$.
The {\em horizontal surfaces} $F_h$ is the horizontal part $P_h$ of $P$ in $E(K)$.
Then we say that $F_v$ and $F_h$ are obtained from $P$ by the {\em $+$-smoothing}, and that $K$ is obtained from $P'$ by the {\em $+$-smoothing} (Figure \ref{smoothing}). 
The {\em $-$-smoothing} of $P$ can be similarly defined, and the results for the $+$-smoothing also hold for the $-$-smoothing.
We note that $F_h$ is always orientable since $P_h$ is closed surfaces in $S^3$, however, $F_v$ is non-orientable in almost all cases.

\begin{figure}[htbp]
	\begin{center}
	\includegraphics[trim=0mm 0mm 0mm 0mm, width=.6\linewidth]{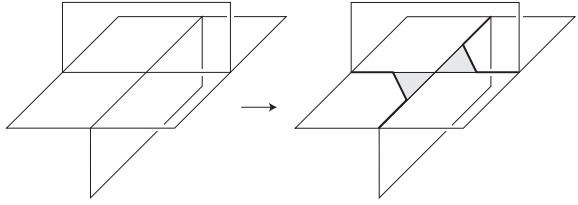}
	\end{center}
	\caption{$+$-smoothing of a closed fake surface $P$}
	\label{smoothing}
\end{figure}

\subsection{Definition of essential closed fake surfaces}
Let $P$ be a closed fake surface embedded in the 3-sphere $S^3$ with the vertical-horizontal decomposition $P=P_v\cup P_h$.
A loop $l$ properly embedded in $P-P'$ is {\em inessential} in $P$ if $l$ bounds a disk $\delta$ in $P-P'$, and $l$ is {\em essential} if it is not inessential.
Let $(\alpha,\partial\alpha)$ be an arc properly embedded in $(P_v, P'-P'')$ or $(P_h, P'-P'')$.
An arc $\alpha$ is {\em inessential} in $P$ if there is an arc $\beta$ in $P'-P''$ such that $\alpha\cup\beta$ bounds a disk in $P_v$ or $P_h$.
Let $D$ be a disk embedded in $S^3$ such that $D\cap P=\partial D\cap (P-P')=\partial D$.
We say that $D$ is a {\em compressing disk} for $P$ if $\partial D$ is essential in $P$.
We say that $D$ is a {\em monogon} if $\partial D\subset P_h-P''$ and $|\partial D\cap P'|=1$.
We say that $D$ is a {\em bigon} if the boundary of $D$ is decomposed into two arcs $\alpha\subset P_v$ and $\beta\subset P_h$ and at least one of $\alpha$, $\beta$ is an essential arc in $P$.
A closed fake surface $P$ embedded in $S^3$ is said to be {\em essential} if 
\begin{enumerate}
\item $S^3-P$ is irreducible, 
\item $P$ has no compressing disk, 
\item $P$ has no monogon, 
\item $P$ has no bigon, and 
\item $P_h$ has no 2-sphere component.
\end{enumerate}

\begin{figure}[htbp]
	\begin{center}
	\begin{tabular}{cc}
	\includegraphics[trim=0mm 0mm 0mm 0mm, width=.25\linewidth]{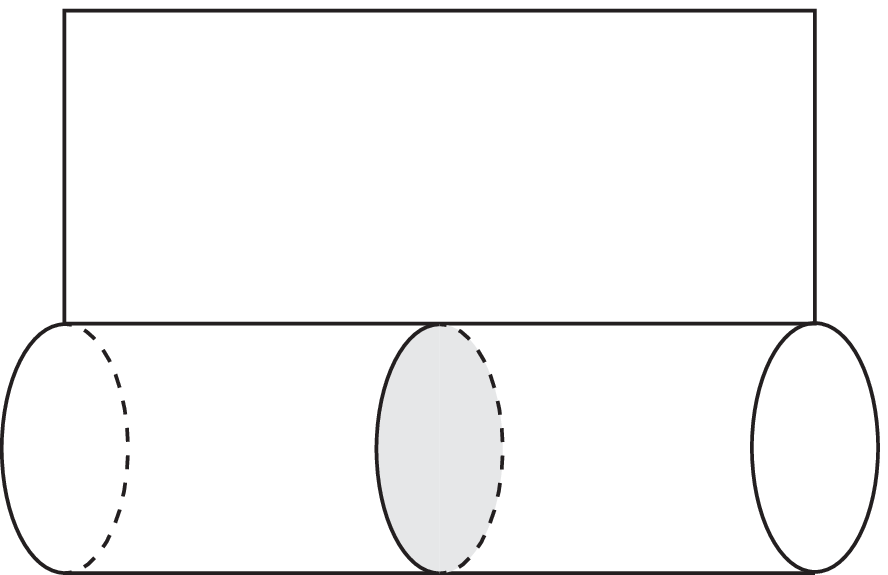} &
	\includegraphics[trim=0mm 0mm 0mm 0mm, width=.4\linewidth]{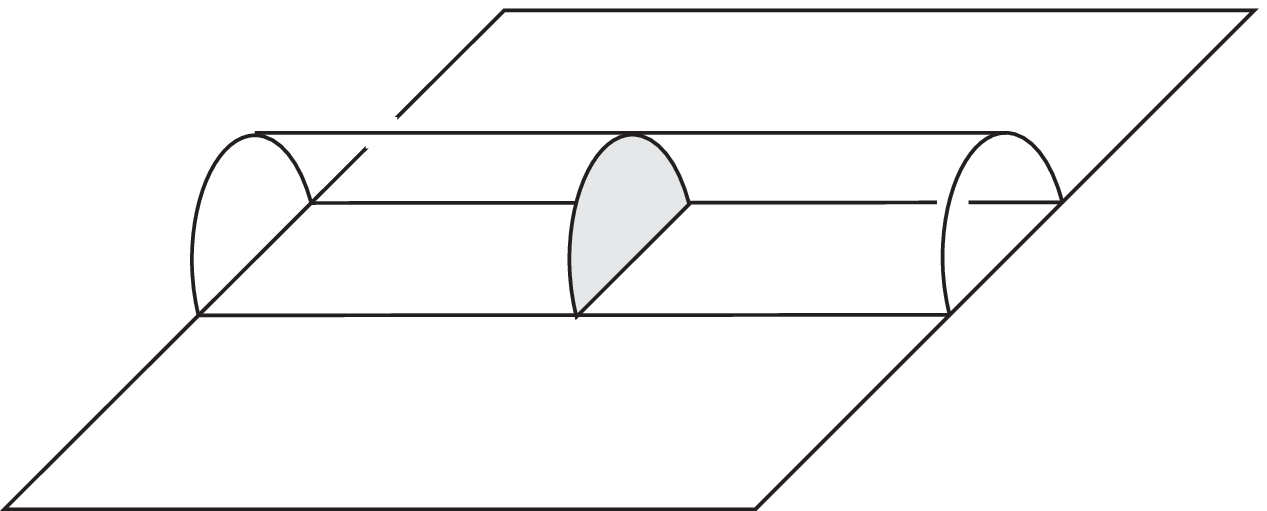} \\
	monogon & bigon 
	\end{tabular}
	\end{center}
	\caption{A monogon and bigon for $P=P_v\cup P_h$}
	\label{trivial}
\end{figure}

\subsection{Definition of essential surfaces}

Let $K$ be a knot or link in $S^3$ and $E(K)$ denote the exterior of $K$.
Let $F$ be a surface properly embedded in $E(K)$, possibly with boundary, except for the 2-sphere or disk, and let $i$ denote the inclusion map $F\to E(K)$.
We say that $F$ is {\em algebraically incompressible} if the induced map $i_*:\pi_1(F)\to \pi_1(E(K))$ is injective, and that $F$ is {\em algebraically boundary incompressible} if the induced map $i_*:\pi_1(F,\partial F)\to \pi_1(E(K),\partial E(K))$ is injective for every choice of two base points in $\partial F$.

A disk $D$ embedded in $E(K)$ is a {\em compressing disk} for $F$ if $D\cap F=\partial D$ and $\partial D$ is an essential loop in $F$.
A disk $D$ embedded in $E(K)$ is a {\em boundary compressing disk} for $F$ if $D\cap F\subset \partial D$ is an essential arc in $F$ and $D\cap \partial E(K)=\partial D-\rm{int}(D\cap F)$.
We say that $F$ is {\em geometrically incompressible} (resp. {\em geometrically boundary incompressible}) if there exists no compressing disk (resp. boundary compressing disk) for $F$.

In this paper, we say that surfaces $F$ embedded in $E(K)$ are {\em geometrically essential} (resp. {\em algebraically essential}) if each component of $F$ is geometrically (resp. algebraically) incompressible, geometrically (resp. algebraically) boundary incompressible and not boundary parallel.
In general, $F$ is algebraically essential if and only if $\partial N(F)\cap E(K)$ is geometrically essential.
If $F$ is two-sided in $E(K)$, namely orientable, then $F$ is algebraically essential if and only if it is geometrically essential.

Let $K$ be a knot or link in $S^3$, and $F$ be closed surfaces embedded in $S^3$.
We say that $F$ is {\em coiled surfaces} for $K$ if $K\subset F$ and $F$ is geometrically essential in the exterior $E(K)$.
We say that a coiled surface $F$ is a {\em Neuwirth surface} if $F-C$ is connected for each component $C$ of $K$.
Surfaces $S$ embedded in $S^3$ is {\em spanning surfaces} for $K$ if $\partial S=K$, and usually $S$ has no closed component.

We remark that any non-trivial, non-splittable knot or link has coiled surfaces since it bounds geometrically incompressible Seifert surfaces $F$ and $\partial N(F)$ gives coiled surfaces.
Similarly, if a knot bounds an algebraically incompressible and boundary incompressible spanning surface $F$, then $\partial N(F)$ gives a Neuwirth surface. 

\subsection{Main theorem}

\begin{theorem}\label{main}
Suppose that $P$ is an essential orientable closed fake surface embedded in the 3-sphere $S^3$ with a vertical-horizontal decomposition $P=P_v\cup P_h$.
Let $F_v$ and $F_h$ be the vertical surfaces and the horizontal surfaces respectively obtained from $P$ by the $+$-smoothing, and $K$ be the knot or link obtained from $P'$ by the $+$-smoothing.
Then $F_v$ and $F_h$ are algebraically essential in $E(K)$ and $K$ is non-splittable and prime.
Moreover when $K$ is a knot, we have that $\Delta(\partial F_v, \partial F_h)=|P''|$ and if $F_v$ is orientable, then $F_h$ is connected.
\end{theorem}

We say that a knot or link $K$ is {\em uniformly twisted} if it can be obtained from $P'$ of an essential orientable closed fake surface $P$ embedded in $S^3$ with a vertical-horizontal decomposition $P=P_v\cup P_h$ by the $+$-smoothing or $-$-smoothing.

In Theorem \ref{main}, if $F_v$ is non-orientable, then $K$ can be put on $\partial N(F_v)$ non-separatively.
Otherwise, $F_h$ is a Neuwirth surface for $K$ by Theorem \ref{main}.
Hence, we have the following corollary.

\begin{corollary}
A uniformly twisted knot satisfies the Neuwirth conjecture.
\end{corollary}

\section{Proof}
The proof is straightforward, but the ``uniformly smoothing" is a key point.

\begin{proof}[Proof of Theorem \ref{main}]
First we isotope $F_v$ near points of Type 3 so that $F_v$ intersects $F_h$ in arcs, which form $\{(x,y,z)||x|\le 1, y=z=0\}$ in the neighborhoods of that points.
Since $F_h$ is orientable and $F_v$ is possibly non-orientable, we need to show that $F_h$ and the (twisted) $\partial I$-bundle $F_v\tilde{\times} \partial I$ are geometrically incompressible and boundary incompressible in $E(K)$.
Then we may assume that in each neighborhood of points of Type 3, $F_h\cap (F_v\tilde{\times} \partial I)$ consists of two arcs.

Suppose that $F_h$ or $F_v\tilde{\times} \partial I$ is compressible in $E(K)$ and let $D$ be a compressing disk for it.
Note that $D$ is in the outside of $F_v\tilde{\times} \partial I$ since $F_v\tilde{\times} \partial I$ is incompressible in $F_v\tilde{\times} I$.
We take $D$ so that $|D\cap (F_v\cup F_h)|$ is minimal.
If $D\cap (F_v\cup F_h)=\emptyset$, then $D$ can be extended to a compressing disk for $P$.
Otherwise, let $\alpha$ be an outermost arc in $D$ and $\delta$ be the corresponding outermost disk of $D$.
We extend $\delta$ so that $\partial \delta\subset F_v\cup F_h$.
Then there are three possibilities, here we note that $\partial \alpha\subset F_v\cap F_h$.

\begin{description}
\item[Case 1] $\alpha$ connects two different arcs of $F_v\cap F_h$ which come from distinct points of Type 3.
\item[Case 2] $\alpha$ connects a single arc of $F_v\cap F_h$ which comes from a single point of Type 3, and $\delta$ lies in a same side of $F_v$ near the point.
\item[Case 3] $\alpha$ connects a single arc of $F_v\cap F_h$ which comes from a single point of Type 3, and $\delta$ lies in both sides of $F_v$ near the point.
\end{description}

In Case 1, $\delta$ can not be trivial since there are two arcs of $P'-P''$ in both sides of $\delta$ (Figure \ref{outermost}).
Here we remember that $F_v$ and $F_h$ are obtained by the $+$-smoothing.
Hence $\delta$ gives a bigon for $P$.
In Case 2, $\delta$ is non-trivial since $|D\cap (F_v\cup F_h)|$ is taken to be minimal.
Hence $\delta$ also gives a bigon for $P$.
Case 3 does not occur since $\partial \delta$ does not run over another side of $F_v$ since $P_v$ is orientable.
Hence $F_h$ and $F_v\tilde{\times} \partial I$ are incompressible in $E(K)$.

\begin{figure}[htbp]
	\begin{center}
	\includegraphics[trim=0mm 0mm 0mm 0mm, width=.6\linewidth]{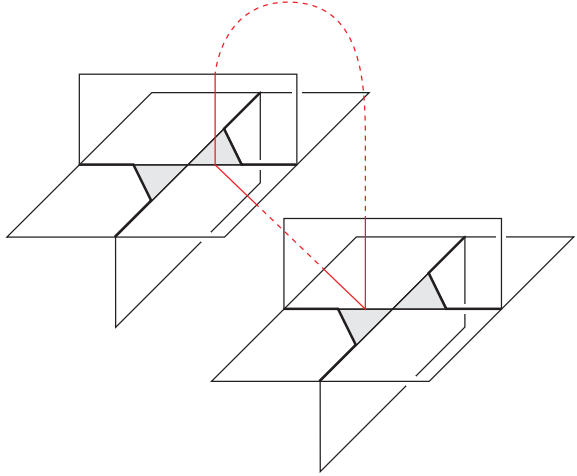}
	\end{center}
	\caption{The boundary of an outermost disk (Case 1)}
	\label{outermost}
\end{figure}

In this stage, we can show that $K$ is non-splittable and non-trivial as follows.
Let $S$ be an essential 2-sphere in $S^3-K$.
By the incompressibility of $F_h$, we may assume that $S$ is disjoint from $F_h$.
Moreover, since $P_v$ is incompressible in the complements of $F_h$, we may assume that $S$ is also disjoint from $P_v$.
Then $S$ bounds a 3-ball in $S^3-K$ since $S^3-P$ is irreducible. 
Hence $K$ is non-splittable.
Suppose that $K$ is trivial.
Then $K$ is a trivial knot since $K$ is non-splittable.
This shows that $P_h$ consists of a single 2-sphere or a single torus since an orientable incompressible surface $F_h$ in a solid torus $E(K)$ is a disk or annulus.
In the former case, it contradicts that $P_h$ has no 2-sphere component.
In the latter case, $F_h$ is an unknotted torus in $S^3$ which bounds a solid torus $V$, and $K$ winds $V$ exactly once.
Then $F_v\cap V$ consists of meridian disks or boundary parallel annuli.
If $F_v\cap V$ consists of meridian disks, then $V$ contains a monogon for $P$, and otherwise $V$ contains a bigon for $P$.
In any case, we have a contradiction.
Hence $K$ is non-trivial.

Next, suppose that $F_h$ is boundary compressible in $E(K)$.
Since it is well-known that a geometrically incompressible, but geometrically boundary compressible orientable surface in a link exterior is a boundary parallel annulus (c.f. \cite[Lemma 2]{O1}), there exists a solid torus $V$ bounded by a component $T$ of $P_h$ such that the component of $K$ contained in $T$ winds around $V$ exactly once longitudinally.
Since $P_v\cap V$ is incompressible in $V$, it consists of meridian disks, or boundary compressible annuli.
If $P_v\cap V$ consists of meridian disks, then the remaining components of $P_v$ having the boundary in $\partial V$ winds around $V$ exactly once longitudinally.
Therefore there exists a monogon for $P$ in $V$.
Otherwise, there exists a bigon for $P$ in $V$ coming from a boundary compressing disk for $P_v\cap V$.
Hence $F_h$ is incompressible and boundary incompressible in $E(K)$.

Suppose that $F_v\tilde{\times} \partial I$ is boundary compressible in $E(K)$.
Since it is well-known that an algebraically incompressible, but algebraically boundary compressible non-orientable spanning surface for a link is a M\"{o}bius band whose boundary is the trivial knot (c.f. \cite[Lemma 2.2]{OT}), $K$ is the trivial knot.
This contradicts that $K$ is non-trivial.
%This contradicts that $F_h$ is incompressible and boundary incompressible in $E(K)$ since an incompressible and boundary incompressible surface in the trivial knot exterior is a disk (c.f. \cite[Lemma]{O1}).
%Hence $F_v\tilde{\times} \partial I$ is incompressible and boundary incompressible in $E(K)$.

Now we know that both $F_h$ and $F_v$ are not boundary parallel in $E(K)$ since these surfaces are incompressible, boundary incompressible and have integral boundary slopes in $E(K)$.
Hence $F_v$ and $F_h$ are algebraically essential in $E(K)$.

In the following, we show that $K$ is prime.
Suppose that $K$ is non-prime and let $S$ be a decomposing sphere for $K$.
We may assume that $S$ intersects $F_h$ in two arcs which form a loop $l$ with two points $p_1$ and $p_2$ of $K\cap S$.
Then $l$ decomposes $S$ into two disks, say $D_1$ and $D_2$.
By an isotopy, we may assume that $\partial D_i$ does not run over neighborhoods of $P''$ for $i=1,\ 2$.
Since there is an arc of $D_i\cap P_v$ from a point $p_j$ for $j=1,\ 2$, there is an arc joining $p_1$ and $p_2$ in $D_1$ or $D_2$. 
Thus we may assume without loss of generality that $D_1$ intersects $P_v$ in a single arc joining $p_1,\ p_2$ and $D_2$ does not intersect $P_v$ in its interior.
Then we have a bigon for $P$ as the subdisk of $D_1$.
Hence $K$ is prime.

Hereafter, we assume that $K$ is a knot.
Then $F_h$ consists of a single closed surface.
It can be observed that the boundary of $F_v$ and $F_h$ intersects in $|P''|$ points essentially on $\partial N(K)$ since $F_v$ and $F_h$ are obtained by the $+$-smoothing (Figure \ref{slope}).
Thus the distance $\Delta(\partial F_v, \partial F_h)$ is equal to $|P''|$.

\begin{figure}[htbp]
	\begin{center}
	\includegraphics[trim=0mm 0mm 0mm 0mm, width=.6\linewidth]{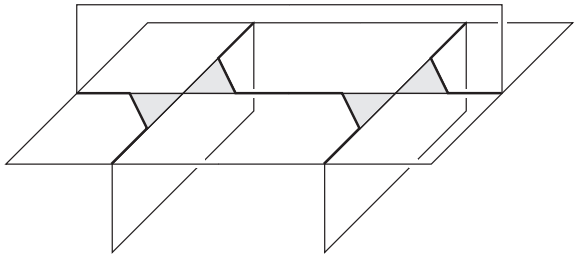}
	\end{center}
	\caption{$+$-smoothing of a closed fake surface $P$}
	\label{slope}
\end{figure}

Suppose that $F_v$ is orientable.
Then $K$ can be oriented by the orientation of $F_v$.
This shows that $F_h$ is connected since $F_v$ and $F_h$ are obtained by the $+$-smoothing (Figure \ref{orientable}).
\begin{figure}[htbp]
	\begin{center}
	\includegraphics[trim=0mm 0mm 0mm 0mm, width=.6\linewidth]{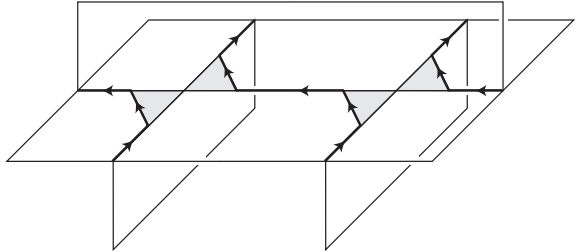}
	\end{center}
	\caption{An orientation of $K$ induced by $F_v$}
	\label{orientable}
\end{figure}
\end{proof}

\section{Example}

In this section, we observe that torus links and alternating links are uniformly twisted, which are typical examples for Theorem \ref{main}.

Let $S^2$ be a 2-sphere embedded in $S^3$ and $G$ be a 2-connected graph embedded in $S^2$ with at least one edge.
Then the closed surface $P_h=\partial N(G)$ decomposes $S^3$ into two handlebodies $V_1$ and $V_2$, where $V_1$ contains $G$.
For each edge $G$, we take parallel copies of a meridian disk of $V_1$ which is dual to an edge of $G$, and
for each region of $S^2-int V_1$, we take parallel copies of a meridian disk of $V_2$ as the region.
Let $P_v$ be a union of these meridian disks.
Then we obtain an essential orientable closed fake surface $P$ with the vertical-horizontal decomposition $P=P_v\cup P_h$.
Let $F_v$ and $F_h$ be the vertical surfaces and the horizontal surfaces respectively obtained from $P$ by the $+$-smoothing or $-$-smoothing, and $K$ be the knot or link obtained from $P'$ by the $+$-smoothing or $-$-smoothing.
%We say that the resultant knot or link is {\em uniformly twisted}.

\begin{figure}[htbp]
	\begin{center}
	\begin{tabular}{ccc}
	\includegraphics[trim=0mm 0mm 0mm 0mm, width=.2\linewidth]{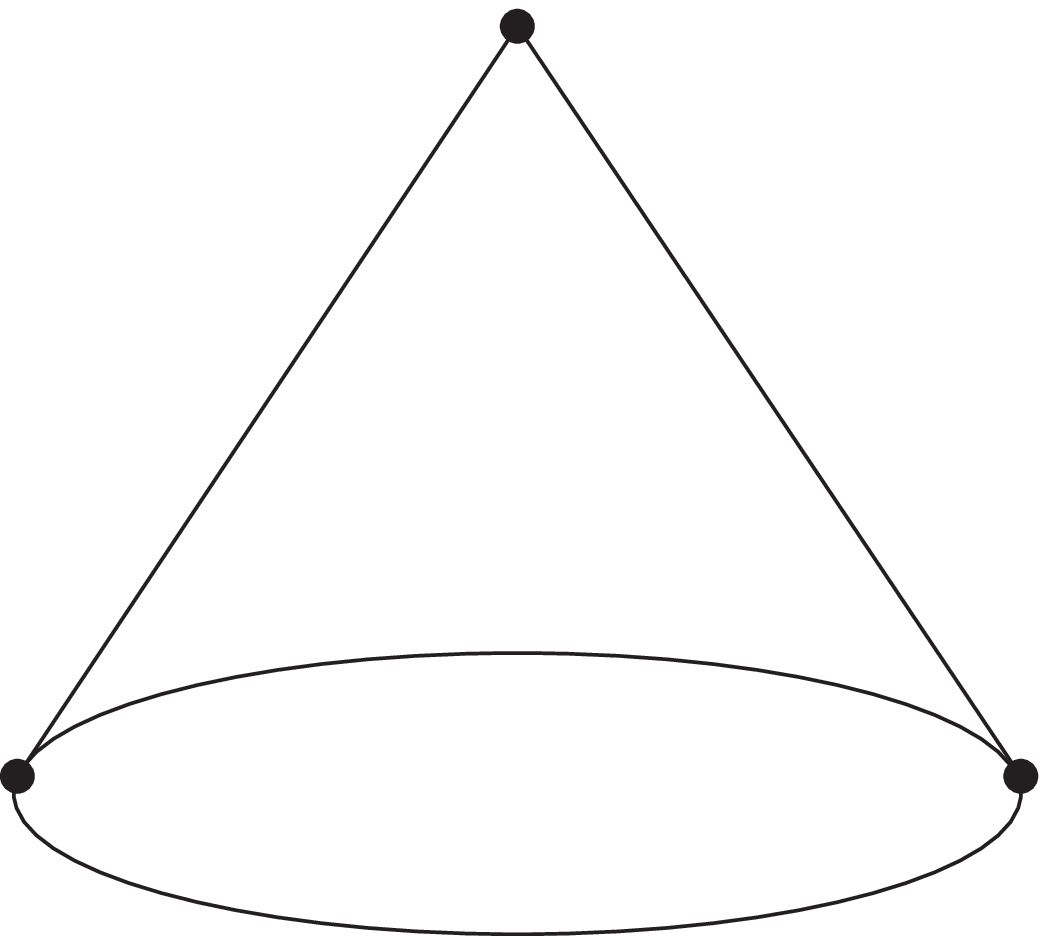} &
	\includegraphics[trim=0mm 0mm 0mm 0mm, width=.35\linewidth]{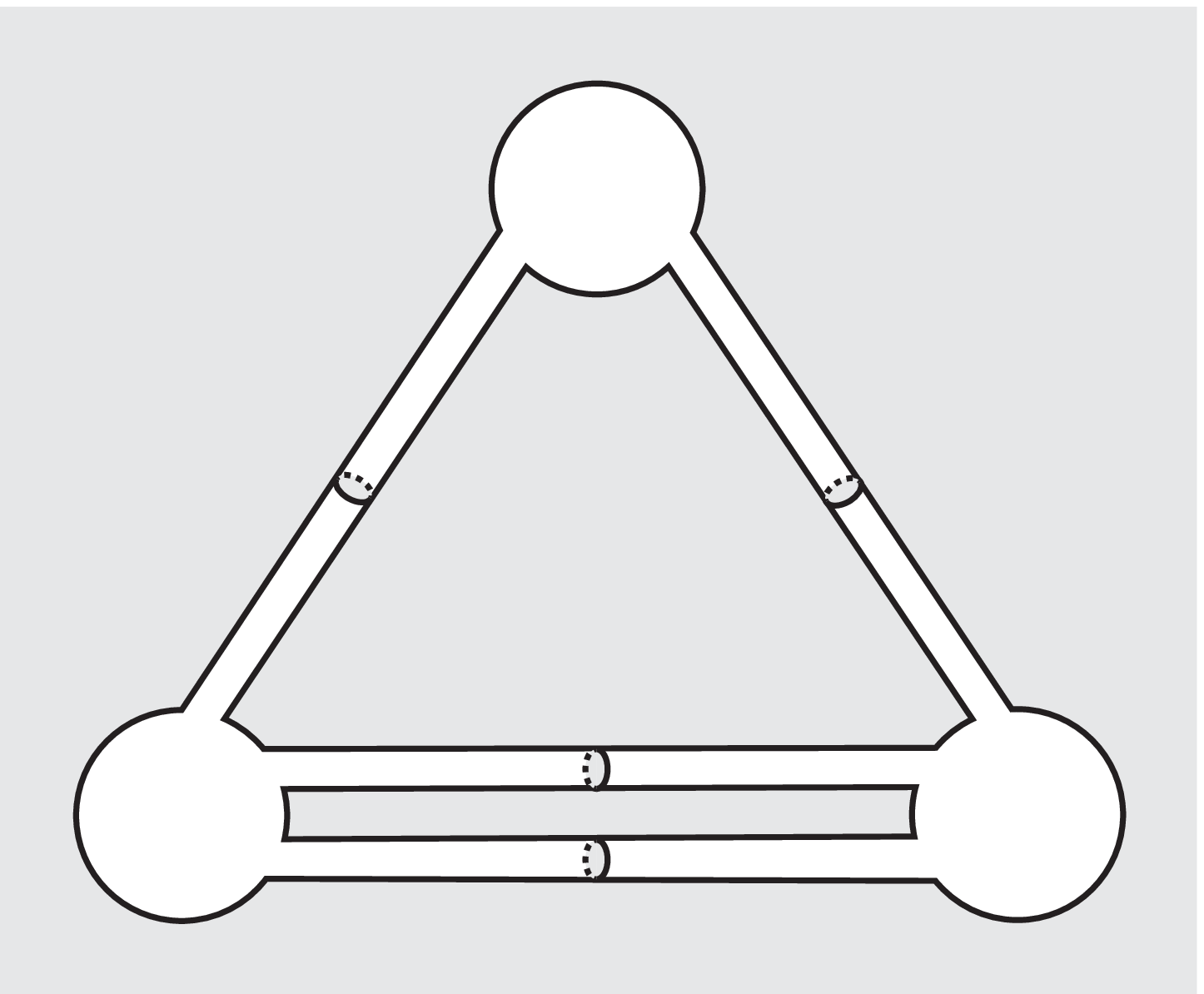} &
	\includegraphics[trim=0mm 0mm 0mm 0mm, width=.33\linewidth]{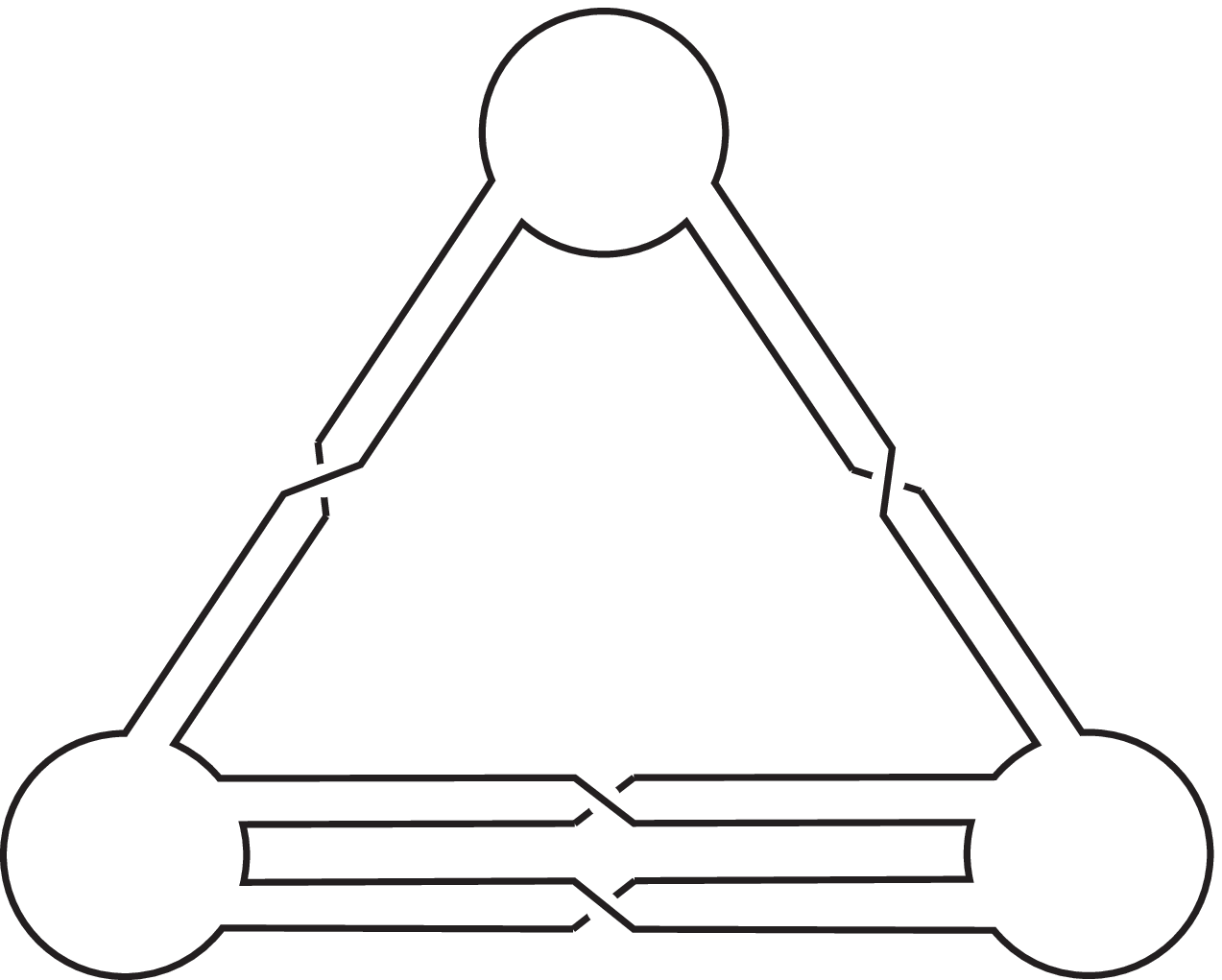} \\
	$2$-connected graph $G$ & closed fake surface $P$ & uniformly twisted knot $K$
	\end{tabular}
	\end{center}
	\caption{The procedure for obtaining a uniformly twisted knot}
	\label{alternating}
\end{figure}

By the construction, torus knots and links are uniformly twisted.
In the above construction, if we take exactly one copy of a meridian disk in $V_1$ or $V_2$, then we obtain a prime alternating knot or link as $K$.
In this case, $F_v$ is a checkerboard surface for the alternating diagram of $K$ and $F_h$ is a boundary of a regular neighborhood of another checkerboard surface.
Similarly, we note that generalized alternating knots and links \cite{O1} are also uniformly twisted.

\section{Problem}

We close this paper by some questions and problems.

In the last section, we construct an essential closed fake surface from a trivial closed surface in the 3-sphere.
We want to know all essential closed fake surfaces obtained from a given closed surface as their horizontal surface.

\begin{problem}
Find a construction of an essential closed fake surface from a given closed surface in the 3-sphere.
\end{problem}

The class of uniformly twisted knots and links is somewhat wide, but it has a restriction ``uniformly twisted''.
We want to know how this class is wide.

\begin{problem}
Does there exist a knot or link which is not uniformly twisted?
\end{problem}

The author think that the answer is yes on this problem since ``$\pm$-smoothing" is a strong condition.
It seems that the condition in Theorem \ref{main} can be weakened.

\begin{problem}
Find a weaker condition on a closed fake surface which gives still the same conclusion of Theorem \ref{main}. 
\end{problem}

The author would expect that all knots and links can be obtained from a closed fake surface which satisfies the weakend condition.

In general, for a knot satisfying the Neuwirth conjecture, does there exist an essential closed fake surface which gives the knot?

\begin{problem}
Let $K$ be a knot with a Neuwirth surface $F_h$, and $F_v$ be an algebraically essential spanning surface for $K$.
Then does there exists an essential closed fake surface $P$ with a vertical-horizontal decomposition $P=P_v\cup P_h$ such that $F_h$ and $F_v$ are obtained from $P$ by the $\pm$-smoothing and $K$ is obtained from $P'$ by the $\pm$-smoothing?
\end{problem}

Finally we extend the Neuwirth conjecture to a link case.

\begin{conjecture}
For any non-trivial, non-split link $L$, there exists a closed surface $F$ containing $L$ such that each component of $L$ is non-separating in $F$ and $F$ is essential in $E(L)$.
\end{conjecture}

\bigskip

\noindent{\bf Acknowledgements.}
I would like to thank Yuya Koda for the careful reading and the detailed report.

\bibliographystyle{amsplain}

\end{document}